\documentclass[english]{elsarticle}
\usepackage[T1]{fontenc}
\usepackage[latin9]{inputenc}
\usepackage[active]{srcltx}
\usepackage{fancybox}
\usepackage{calc}
\usepackage{units}
\usepackage{amsthm}
\usepackage{amsmath}
\usepackage{amssymb}
\usepackage[all]{xy}
\PassOptionsToPackage{normalem}{ulem}
\usepackage{ulem}

\makeatletter

\providecommand{\tabularnewline}{\\}

\theoremstyle{plain}
\newtheorem{thm}{\protect\theoremname}
\theoremstyle{definition}
\newtheorem{defn}[thm]{\protect\definitionname}
\theoremstyle{plain}
\newtheorem{prop}[thm]{\protect\propositionname}
\theoremstyle{definition}
\newtheorem{example}[thm]{\protect\examplename}
\ifx\proof\undefined
\newenvironment{proof}[1][\protect\proofname]{\par
\normalfont\topsep6\p@\@plus6\p@\relax
\trivlist
\itemindent\parindent
\item[\hskip\labelsep\scshape #1]\ignorespaces
}{%
\endtrivlist\@endpefalse
}
\providecommand{\proofname}{Proof}
\fi
\theoremstyle{remark}
\newtheorem{rem}[thm]{\protect\remarkname}
\theoremstyle{plain}
\newtheorem{cor}[thm]{\protect\corollaryname}

\@ifundefined{date}{}{\date{}}
\usepackage[all]{xy} 
\usepackage{etex}
\xyoption{2cell}
\usepackage{afterpage}
\usepackage{pdflscape}
\usepackage{multicol}
\usepackage[pdftex]{hyperref}
\usepackage{fancyhdr}
\usepackage{setspace}
\usepackage{tikz}
\usetikzlibrary{matrix,arrows}
\usepackage{proof}
\date{}

\makeatother

\usepackage{babel}
\providecommand{\corollaryname}{Corollary}
\providecommand{\definitionname}{Definition}
\providecommand{\examplename}{Example}
\providecommand{\propositionname}{Proposition}
\providecommand{\remarkname}{Remark}
\providecommand{\theoremname}{Theorem}

\begin{document}

\begin{frontmatter}{}

\title{\textbf{\huge{}Categorical Comprehensions and Recursion}}

\author[rvt]{Joaqu\'in D\'iaz Boils\corref{cor1}}
 \cortext[cor1]{corresponding author at: boils@uji.es}  \address[rvt]{Facultad de Ciencias Exactas y Naturales.

Pontificia Universidad Cat\'olica del Ecuador. 

170150. Quito. Ecuador.}
\begin{abstract}
{\normalsize{}A new categorical setting is defined in order to characterize
the subrecursive classes belonging to complexity hierarchies. This
is achieved by means of }\emph{\normalsize{}coercion functors}{\normalsize{}
over a symmetric monoidal category endowed with certain recursion
schemes that imitate the }\emph{\normalsize{}bounded recursion scheme}{\normalsize{}.
This gives a categorical counterpart of generalized }\emph{\normalsize{}safe
composition}{\normalsize{} and }\emph{\normalsize{}safe recursion}{\normalsize{}.}{\normalsize \par}\end{abstract}

\end{frontmatter}{}

\section{Introduction}

Various recursive function classes have been characterized in categorical
terms. It has been achieved by considering a category with certain
structure endowed with a recursion scheme. The class of Primitive
Recursive Functions ($\mathcal{PR}$ in the sequel), for instance,
has been \emph{chased} simply by means of a cartesian category and
a \emph{Natural Numbers Object} \emph{with parameters} (\emph{nno}
in the sequel, see \cite{L. Rom=0000E1n}). In \cite{Rom=0000E1n-Par=0000E9}
it can be found a generalization of that characterization to a monoidal
setting, that is achieved by endowing a monoidal category with a special
kind of nno (a \emph{left nno}) where the tensor product is included.
It is also known that other classes containing $\mathcal{PR}$ can
be obtained by adding \emph{more structure}: considering for instance
a topos (\cite{Lambek-Scott}), a cartesian closed category (\cite{Thibault})
or a category with finite limits (\citep{L. Rom=0000E1n[2]}).%
\footnote{See \cite{Zalamea} for a summary of those results.%
}

Less work has been made, however, on categorical characterizations
of \emph{subrecursive function classes}, that is, those contained
in $\mathcal{PR}$ (see \citep{Cockett-Diaz-Gallagher-Hrubes} and
\citep{Cockett-Redmond}). In $\mathcal{PR}$ there is at least a
sequence of functions such that every function in it has a more complex
growth than the preceding function in the sequence. Such function
scale allows us to define a \emph{hierarchy} in $\mathcal{PR}$ with
which we can classify the primitive recursive functions according
to its \emph{level of complexity}. This is the case of the \emph{Grzegorzcyk}
\emph{Hierarchy}. 

A reason to not have more studies of subrecursive function classes
in Category Theory at our disposal is that we lack a recursive diagram
with enough expressiveness to characterize the operation of \emph{bounded
recursion} under which most of those classes are closed and looking
like

\[
\begin{cases}
 & f(u,0)=g(u)\\
 & f(u,x+1)=h(u,x,f(u,x))\\
 & f(u,x)\leq j(u,x)
\end{cases}
\]

\noindent The problem arises when, given those functions $g$, $h$
and $j$, we want to know if there exists a function $f$ satisfying
the three conditions in the bounded recursion scheme. 

The known as \emph{safe recursion scheme} was introduced by Bellantoni
and Cook in \citep{Bellantoni-Cook} as a way to substitute the bounding
condition in the above scheme by a syntactical condition. The central
idea of S. Bellantoni and S. Cook was to define two different kinds
of variables (\emph{normal} and \emph{safe variables}) according to
the use we make of them in the process of computation (see \citep{Bellantoni-Cook}
for more details). In \citep{Bellantoni-Cook} the class of polynomial
time functions has been characterized and, subsequently, several other
subrecursive classes.

The \emph{ramified recursion}, in turn, is a way to avoid \emph{impredicativity}
problems. In a ramified system the objects are defined using levels
such that the definition of an object in level $i$ depends only on
levels below $i$. According to \cite{Leivant}, by considering recursion
over a word algebra $\mathbf{A}$, we can get a collection of \emph{levels}
$\mathbf{A}_{j}$ of $\mathbf{A}$ seen as types or universes where
everyone of them contains a copy of the constructors.

The method we will use consists in considering a collection of copies
of $\mathbf{N}$, denoted by $\mathbf{N}_{k}$, such that the functions
defined in every (isomorphic) copy are: 
\begin{itemize}
\item in $\mathbf{N}_{0}$ certain initial functions where \emph{zero} and
\emph{successor} are always present
\item in $\mathbf{N}_{k+1}$ the definable functions using functions defined
in $\mathbf{N}_{j}$ with $j\leq k$ and certain operators, among
which are recursion operators, and whose recursion has been made over
values in $\mathbf{N}_{s}$ with $s\leq k$.
\end{itemize}
We will call these $\mathbf{N}_{i}$ \emph{levels of the natural numbers
}and they have a close relation with different function classes according
to its complexity level. 

The thesis \cite{Otto} uses categories of ordinal numbers%
\footnote{Hereafter we will only consider finite ordinals.%
} to define \emph{coercion functors} with the idea of chasing the ramification
conditions of \cite{Leivant}. Using this method, and introducing
the concept of \emph{symmetric monoidal 2-} and \emph{3-Comprehensions},
J. R. Otto tries to characterize several subrecursive function classes
such as linear time, polynomial time, polynomial space and the classes
$\mathcal{E}^{2}$ and $\mathcal{E}^{3}$ of \emph{Grzegorzcyk} \emph{Hierarchy.} 

The aim of this paper is to give a categorical characterization of
subrecursive hierarchies based on the operations of safe recursion
and composition.

\section{Basic structures}
\begin{defn}
For each $n\in\mathbf{N}$ the category $\mathbf{n}$ has as objects
the natural numbers lower than $n$ and as arrows 
\[
0\longrightarrow1\longrightarrow\cdots\longrightarrow n-1
\]
corresponding to the order of $n$. We denote by $m_{i,j}$ the only
arrow from $i$ to $j$ with $0\leq i\leq j<n$.
\end{defn}

\begin{defn}
Let $M_{n}^{op}$ be the monoid of endofunctors in $\mathbf{n}$ in
which the product $fg$ is the composition $g\circ f$.%
\footnote{$M_{n}^{op}$ is exactly the set of monotone functions from $(n,\leq)$
to $(n,\leq)$ with $n=\{0,...,n-1\}$.%
}

Let's establish a set of elements in $M_{n}^{op}$ from which one
can generate the rest of elements by means of multiplication. This
set is used in \cite{Otto} in the case of $n=2$ and $n=3$.

Let be for every $0\leq k<n-1$ the functors $id:\mathbf{n}\longrightarrow\mathbf{n}$,
$T_{k}:\mathbf{n}\longrightarrow\mathbf{n}$ and $G_{k}:\mathbf{n}\longrightarrow\mathbf{n}$
such that for all $j\in\mathbf{n}$:

\[
id(j)=j\qquad T_{k}(j)=\left\{ \begin{array}{ll}
k+1 & \text{if }j=k\\
j & \text{if }j\neq k
\end{array}\right.\qquad G_{k}(j)=\left\{ \begin{array}{ll}
k & \text{if }j=k+1\\
j & \text{if }j\neq k+1
\end{array}\right.
\]
taking the form

{\small{}
\[
[T_{0}]\;\xymatrix{0\ar[dr] & 0\\
1\ar[r] & 1\\
2\ar[r] & 2\\
\vdots & \vdots\\
n-1\ar[r] & n-1
}
\textrm{ }\textrm{ }[T_{1}]\;\xymatrix{0\ar[r] & 0\\
1\ar[dr] & 1\\
2\ar[r] & 2\\
\vdots & \vdots\\
n-1\ar[r] & n-1
}
\textrm{ }\textrm{ }\textrm{ }\textrm{ }\cdots\textrm{ }\textrm{ }[T_{n-2}]\;\xymatrix{0\ar[r] & 0\\
1\ar[r] & 1\\
\vdots & \vdots\\
n-2\ar[dr] & n-2\\
n-1\ar[r] & n-1
}
\]
}and

{\small{}
\[
[G_{0}]\;\xymatrix{0\ar[r] & 0\\
1\ar[ur] & 1\\
2\ar[r] & 2\\
\vdots & \vdots\\
n-1\ar[r] & n-1
}
\textrm{ }\textrm{ }[G_{1}]\;\xymatrix{0\ar[r] & 0\\
1\ar[r] & 1\\
2\ar[ur] & 2\\
\vdots & \vdots\\
n-1\ar[r] & n-1
}
\textrm{ }\textrm{ }\textrm{ }\textrm{ }\cdots\textrm{ }\textrm{ }[G_{n-2}]\;\xymatrix{0\ar[r] & 0\\
1\ar[r] & 1\\
\vdots & \vdots\\
n-2\ar[r] & n-2\\
n-1\ar[ur] & n-1
}
\]
}In the sequel we will refer to different $T$ and $G$ as \emph{coercion
functors}.\end{defn}
\begin{prop}
\label{Genera}For every $n\in\mathbf{N}$ the monoid $M_{n}^{op}$
can be generated by the finite set 
\[
\{G_{0},\cdots,G_{n-2},T_{0},\cdots,T_{n-2}\}.
\]

\end{prop}
Now we consider some particular natural transformations in $\mathbf{n}$.
\begin{defn}
\label{nat}Let $\epsilon_{k}:G_{k}\Longrightarrow id$ and $\eta_{k}:id\Longrightarrow T_{k}$
($0\leq k\leq n-2$) be such that for $i\in\mathbf{n}$ 
\[
\epsilon_{k}(i)=\left\{ \begin{array}{ll}
m_{i,i} & \text{if }i\neq k+1\\
m_{i-1,i} & \text{if }i=k+1
\end{array}\right.\qquad\eta_{k}(i)=\left\{ \begin{array}{ll}
m_{i,i} & \text{if }i\neq k\\
m_{i,i+1} & \text{if }i=k
\end{array}\right..
\]
\end{defn}
\begin{thm}
Every non-identity natural transformation in $\mathbf{n}$ can be
generated by means of a composition of natural transformations from
Definition \ref{nat} and right and left multiplication of those natural
transformations and functors from Proposition \ref{Genera}.
\end{thm}
$M_{n}^{op}$ can be seen as a category whose objects are the endofunctors
in $\mathbf{n}$ and whose arrows are the natural transformations
in $\mathbf{n}$.
\begin{thm}
For the definitions given above we have the following chain of adjunctions
\[
T_{k}\dashv G_{k}\dashv T_{k+1}\dashv G_{k+1}
\]
for every $k\in\{0,1,...,n-3\}$. 
\end{thm}

\section{\emph{SM n}-Comprehensions}

For the definition of \emph{SM n-Comprehension} in this Section we
need to consider relations among categories allowing the definition
of categorical structures arising from other structures based on certain
properties that the former \emph{inherits} from the latter. A category
will then have the same certain bicategorical property of another
category if the same commutative diagrams are satisfied for them both.
That is, if there exists a bifunctor between them.
\begin{defn}
A \emph{SM n-Comprehension} $(\mathcal{C},\left\langle T_{k}^{\mathcal{C}}\right\rangle ,\left\langle G_{k}^{\mathcal{C}}\right\rangle ,\left\langle \eta_{k}^{\mathcal{C}}\right\rangle ,\left\langle \epsilon_{k}^{\mathcal{C}}\right\rangle )$
consists of
\begin{itemize}
\item A SM category $\mathcal{C}=(\otimes,\top,l,a,\sigma)$,%
\footnote{We omit the introduction of the \emph{right identity} $r:C\otimes\top\rightarrow C$
defined for every object $C$ in $\mathcal{C}$ for being definable
in terms of $\sigma$ and $l$ as $C\otimes\top\overset{\sigma}{\rightarrow}\top\otimes C\overset{l}{\rightarrow}C$.
It will be used elsewhere in the sequel, however. We also express
the objects \emph{modulo associativity and symmetry} in the sequel.%
}
\item for every $k$ such that $0\leq k<n-1$ the \emph{SM functors} $T_{k}^{\mathcal{C}},G_{k}^{\mathcal{C}}:\mathcal{C}\longrightarrow\mathcal{C}$,%
\footnote{By SM functors\emph{ }we understand that\emph{ }$T_{k}^{\mathcal{C}},G_{k}^{\mathcal{C}}:\mathcal{C}\longrightarrow\mathcal{C}$
satisfy: 

\begin{center}
{\footnotesize{}$T_{k}^{\mathcal{C}}\top=G_{k}^{\mathcal{C}}\top=\top$}
\par\end{center}{\footnotesize \par}

\begin{center}
{\footnotesize{}$T_{k}^{\mathcal{C}}(f\otimes Y)=T_{k}^{\mathcal{C}}f\otimes T_{k}^{\mathcal{C}}Y\qquad$
$G_{k}^{\mathcal{C}}(f\otimes Y)=G_{k}^{\mathcal{C}}f\otimes G_{k}^{\mathcal{C}}Y$}
\par\end{center}{\footnotesize \par}

\begin{center}
{\footnotesize{}$T_{k}^{\mathcal{C}}aXYZ=a(T_{k}^{\mathcal{C}}X)(T_{k}^{\mathcal{C}}Y)(T_{k}^{\mathcal{C}}Z)\qquad$
$G_{k}^{\mathcal{C}}aXYZ=a(G_{k}^{\mathcal{C}}X)(G_{k}^{\mathcal{C}}Y)(G_{k}^{\mathcal{C}}Z)$}
\par\end{center}{\footnotesize \par}

\begin{center}
{\footnotesize{}$T_{k}^{\mathcal{C}}\sigma XY=\sigma(T_{k}^{\mathcal{C}}X)(T_{k}^{\mathcal{C}}Y)\qquad$
$G_{k}^{\mathcal{C}}\sigma XY=\sigma(G_{k}^{\mathcal{C}}X)(G_{k}^{\mathcal{C}}Y)$}
\par\end{center}{\footnotesize \par}

\begin{center}
{\footnotesize{}$T_{k}^{\mathcal{C}}lX=lT_{k}^{\mathcal{C}}X\qquad$
$\textrm{ }G_{k}^{\mathcal{C}}lX=lG_{k}^{\mathcal{C}}X$. }
\par\end{center}%
}
\item for every $k$ such that $0\leq k<n-1$ the \emph{SM} \emph{transformations}
$\eta_{k}^{\mathcal{C}}:id\Longrightarrow T_{k}^{\mathcal{C}}$ and
$\epsilon_{k}^{\mathcal{C}}:G_{k}^{\mathcal{C}}\Longrightarrow id$,%
\footnote{By SM functors\emph{ }we understand that $\eta_{k}^{\mathcal{C}}:id\Longrightarrow T_{k}^{\mathcal{C}}$
and $\epsilon_{k}^{\mathcal{C}}:G_{k}^{\mathcal{C}}\Longrightarrow id$
satisfy 

\begin{center}
{\footnotesize{}$\eta_{k}^{\mathcal{C}}\top=\epsilon_{k}^{\mathcal{C}}\top=1_{\top}$,
$\eta_{k}^{\mathcal{C}}(X\otimes Y)=\eta_{k}^{\mathcal{C}}X\otimes\eta_{k}^{\mathcal{C}}Y$
and $\epsilon_{k}(X\otimes Y)=\epsilon_{k}^{\mathcal{C}}X\otimes\epsilon_{k}^{\mathcal{C}}Y$}
\par\end{center}%
}
\end{itemize}

\noindent and the existence of a bifunctor $\Im:M_{n}^{op}\rightarrow(\mathcal{C},\mathcal{C})$
such that $\Im(T_{k})=T_{k}^{\mathcal{C}}$, $\Im(G_{k})=G_{k}^{\mathcal{C}}$,
$\Im(\eta_{k})=\eta_{k}^{\mathcal{C}}$ and $\Im(\epsilon_{k})=\epsilon_{k}^{\mathcal{C}}$.%
\footnote{That is, what we ask is to commute the same diagrams for $T_{k}^{\mathcal{C}}$,
$G_{k}^{\mathcal{C}}$, $\eta_{k}^{\mathcal{C}}$ and $\epsilon_{k}^{\mathcal{C}}$
than $T_{k}$, $G_{k}$, $\eta_{k}$ and $\epsilon_{k}$. For $\Im$
exist we are looking at $M_{n}^{op}$ as a bicategory with a unique
0-cell $\mathbf{n}$.%
}

\end{defn}
We will denote $(\mathcal{C},\left\langle T_{k}\right\rangle ,\left\langle G_{k}\right\rangle ,\left\langle \eta_{k}\right\rangle ,\left\langle \epsilon_{k}\right\rangle )$
for $(\mathcal{C},\left\langle T_{k}^{\mathcal{C}}\right\rangle ,\left\langle G_{k}^{\mathcal{C}}\right\rangle ,\left\langle \eta_{k}^{\mathcal{C}}\right\rangle ,\left\langle \epsilon_{k}^{\mathcal{C}}\right\rangle )$
when there is no ambiguity.

We will now see that an analogous structure can be defined for the
exponential of a category by considering two different starting cases:
a given \emph{SM }n-Comprehension (Exemple 8) or simply a \emph{SM}
category (Exemple 9). We will see for both structures how a sort of
\emph{exponential} \emph{SM n-Comprehension} can be constructed in
quite a different way. This is achieved by using the \emph{cotensor
product of two $\mathcal{V}$-categories}, a concept we recall in
Appendix 1, specialized to the case of $\mathbf{n}\multimap\mathcal{C}$. 
\begin{example}
An exemple of an exponential \emph{SM} n-Comprehension from a given
a \emph{SM n}-Comprehension $(\mathcal{C},\left\langle T_{k}^{\mathcal{C}}\right\rangle ,\left\langle G_{k}^{\mathcal{C}}\right\rangle ,\left\langle \eta_{k}^{\mathcal{C}}\right\rangle ,\left\langle \epsilon_{k}^{\mathcal{C}}\right\rangle )$
is constructed in the following. We denote for $\chi_{k}^{\mathcal{C}}=\eta_{k}^{\mathcal{C}}\circ\epsilon_{k}^{\mathcal{C}}$
a natural transformation from $G_{k}^{\mathcal{C}}$ to $T_{k}^{\mathcal{C}}$
for each $k=0,...,n-2$ and for which we have the obvious equalities
\[
T_{k}^{\mathcal{C}}\epsilon_{k}^{\mathcal{C}}=G_{k}^{\mathcal{C}}\eta_{k}^{\mathcal{C}}=\chi_{k}^{\mathcal{C}}.
\]

We define a functor between $\mathcal{C}$ and $\mathbf{n}\multimap\mathcal{C}$
by taking the following endofunctors in $\mathcal{C}$%
\footnote{Whenever $k=0$ this expression takes the form $G_{n-2}^{\mathcal{C}}...G_{0}^{\mathcal{C}}$
and in the case of $k=n-1$ the form $T_{0}^{\mathcal{C}}...T_{n-2}^{\mathcal{C}}$.%
} 
\[
G_{n-2}^{\mathcal{C}}...G_{k}^{\mathcal{C}}T_{0}^{\mathcal{C}}...T_{k-1}^{\mathcal{C}}
\]
whose corresponding functors in $\mathbf{n}$ give constant values
for $0\leq k\leq n-1$. That is 
\[
G_{n-2}...G_{k}T_{0}...T_{k-1}(j)=k
\]
for every $j=0,...,n-1$.

We denote $\overline{k}$ for $G_{n-2}^{\mathcal{C}}...G_{k}^{\mathcal{C}}T_{0}^{\mathcal{C}}...T_{k-1}^{\mathcal{C}}$
where $0\leq k\leq n-1$. Let $\chi$ be then the assignation 
\[
\begin{array}{c}
\chi(k)=\overline{k}\\
\chi(m_{k,k+1})=\overline{k}\chi_{k}^{\mathcal{C}}
\end{array}
\]
and
\[
\chi(f\circ g)=\chi(f)\circ\chi(g),
\]
for all $k=0,1,...,n-2$ and for every pair of morphisms $f$ and
$g$ in $\mathbf{n}$. We then have the following assignations:

{\footnotesize{}
\[
\xymatrix{ & \mathbf{n}\overset{\chi}{\longrightarrow}\mathcal{SM}(\mathcal{C},\mathcal{C})\\
0\ar[d]_{m_{0,1}} &  & \overline{0}\ar[d]^{\overline{0}\chi_{0}^{\mathcal{C}}}\\
1\ar[d]_{m_{1,2}} &  & \overline{1}\ar[d]^{\overline{1}\chi_{1}^{\mathcal{C}}}\\
2\ar[d]_{m_{2,3}} & \Longrightarrow & \overline{2}\ar[d]^{\overline{2}\chi_{2}^{\mathcal{C}}}\\
\vdots\ar[d]_{m_{n-3,n-2}} &  & \vdots\ar[d]^{\overline{(n-3)}\chi_{n-3}^{\mathcal{C}}}\\
n-2\ar[d]_{m_{n-2,n-1}} &  & \overline{(n-2)}\ar[d]^{\overline{(n-2)}\chi_{n-2}^{\mathcal{C}}}\\
n-1 &  & \overline{(n-1)}
}
\]
}{\footnotesize \par}

We stress here that $\chi_{k}^{\mathcal{C}}:G_{k}^{\mathcal{C}}\Longrightarrow T_{k}^{\mathcal{C}}$
are natural transformations for endofunctors in $\mathcal{C}$ while
$\chi$ can be seen as a bifunctor with domain $\mathbf{n}$, seen
as a bicategory, and $\mathcal{SM}(\mathcal{C},\mathcal{C})$ as codomain. 

For $\mathbf{n}\multimap\mathcal{C}$ functors are chains of natural
transformations in the form 
\[
\overline{k}\chi_{k}
\]
with $0\leq k\leq n-2$. That is, starting from the unique $(n-1)$-tuple
of natural transformations
\[
[\overline{0}\chi_{0},\overline{1}\chi_{1},...,\overline{(n-2)}\chi_{n-2}]
\]
in $\mathcal{SM}(\mathcal{C},\mathcal{C})$, that can be seen as the
assignation of
\[
\chi:\mathbf{n}\rightarrow\mathcal{SM}(\mathcal{C},\mathcal{C})
\]
for constant values in $Cat$, it can be generated another assignation%
\footnote{By considering the isomorphism $\mathcal{SM}(\mathcal{C},\mathbf{n}\multimap\mathcal{C})\cong Cat(\mathbf{n},\mathcal{SM}(\mathcal{C},\mathcal{C}))$
given above.%
}
\[
\overline{\chi}:\mathcal{C}\rightarrow\mathbf{n}\multimap\mathcal{C}
\]
in $\mathcal{SM}$. This new assignation has in the case of a \emph{n-Comprehension},
among others, the form we have introduced above. 

\noindent With this construction we can assert that whenever $(\mathcal{C},\left\langle T_{k}\right\rangle ,\left\langle G_{k}\right\rangle ,\left\langle \eta_{k}\right\rangle ,\left\langle \epsilon_{k}\right\rangle )$
is a \emph{SM} \emph{n-Comprehension} we can construct a new tuple\emph{
\[
(\mathbf{n}\multimap\mathcal{C},\left\langle T_{k}^{\mathbf{n}}\right\rangle ,\left\langle G_{k}^{\mathbf{n}}\right\rangle ,\left\langle \eta_{k}^{\mathbf{n}}\right\rangle ,\left\langle \epsilon_{k}^{\mathbf{n}}\right\rangle )
\]
}being itself a \emph{SM} \emph{n-Comprehension}.
\end{example}

\begin{example}
\label{setn}An exemple of an exponential \emph{SM} n-Comprehension
from a given \emph{SM} category $\mathcal{C}$ is constructed by considering
again $\mathbf{n}\multimap\mathcal{C}$. 

We now define some endofunctors $T^{e}$ and $G^{e}$ acting in such
a way that for every 
\[
X_{0}\overset{h_{0}}{\longrightarrow}...\overset{h_{n-2}}{\longrightarrow}X_{n-1}
\]
we obtain {\footnotesize{}
\[
T_{k}^{e}(X_{0}\rightarrow...\rightarrow X_{n-1})=X_{0}\rightarrow...\rightarrow X_{k-1}\overset{t}{\longrightarrow}X_{k+1}\overset{id}{\longrightarrow}X_{k+1}\rightarrow...\rightarrow X_{n-1}
\]
}and{\footnotesize{}
\[
G_{k}^{e}(X_{0}\rightarrow...\rightarrow X_{n-1})=X_{0}\rightarrow...\rightarrow X_{k}\overset{id}{\longrightarrow}X_{k}\overset{g}{\longrightarrow}X_{k+2}\rightarrow...\rightarrow X_{n-1}
\]
}where $t=h_{k-1}\circ h_{k}$ and $g=h_{k}\circ h_{k+1}$ and for
every chain of \emph{vertical arrows} $(f_{0},...,f_{n-1})$ we obtain%
\footnote{With the notation established in the description of $\mathbf{n}\multimap\mathcal{C}$
in Appendix 1.%
} 
\[
T_{k}^{e}(f_{0},...,f_{n-1})=(f_{0},...,f_{k-1},f_{k+1},f_{k+1},...,f_{n-1})
\]
and
\[
G_{k}^{e}(f_{0},...,f_{n-1})=(f_{0},...,f_{k},f_{k},f_{k+2},...,f_{n-1})
\]
Making of $\mathbf{n}\multimap\mathcal{C}$ a SM n-Comprehension.

Fixing a single object $X$ there are some special objects in the
form\emph{ }
\[
X^{0}=\;\xymatrix{X\ar[d]\\
\top\ar[d]\\
\top\ar[d]\\
\top\ar[d]\\
\overset{\vdots}{\top}
}
\textrm{ }\textrm{ }\textrm{ }\textrm{ }X^{1}=\;\xymatrix{X\ar[d]^{id_{X}}\\
X\ar[d]\\
\top\ar[d]\\
\top\ar[d]\\
\overset{\vdots}{\top}
}
\textrm{ }\textrm{ }X^{2}=\;\xymatrix{X\ar[d]^{id_{X}}\\
X\ar[d]^{id_{X}}\\
X\ar[d]\\
\top\ar[d]\\
\overset{\vdots}{\top}
}
\cdots\textrm{ }\textrm{ }\textrm{ }\textrm{ }X^{n-1}=\;\xymatrix{X\ar[d]^{id_{X}}\\
X\ar[d]^{id_{X}}\\
X\ar[d]^{id_{X}}\\
X\ar[d]^{id_{X}}\\
\overset{\vdots}{X}
}
\]
where the chains are formed by $n$ objects and $n-1$ arrows. We
call these objects the \emph{levels of} $X$. This levels of an object
$X$ can also be generated by applications of the endofunctors $G_{k}^{e}$
starting from $X^{0}$: 
\[
X^{k}=G_{k-1}^{e}X^{k-1}
\]
or else, starting from $X^{n-2}$ and excluding $X^{n-1}$, by 
\[
X^{k}=T_{k+1}^{e}X^{k+1}
\]
when $k=0,...,n-2$. It gives the following table for the levels of
the object $X$

\begin{center}
\begin{tabular}{|c|cccccc}
\cline{2-7} 
\multicolumn{1}{c|}{} & \multicolumn{1}{c|}{$X^{0}$} & \multicolumn{1}{c|}{$X^{1}$} & \multicolumn{1}{c|}{...} & \multicolumn{1}{c|}{$X^{n-3}$} & \multicolumn{1}{c|}{$X^{n-2}$} & \multicolumn{1}{c|}{$X^{n-1}$}\tabularnewline
\hline 
$T_{0}^{e}$ & $1^{n}$ & $X^{1}$ & ... & $X^{n-3}$ & $X^{n-2}$ & $X^{n-1}$\tabularnewline
\cline{1-1} 
$G_{0}^{e}$ & $X^{1}$ & $X^{1}$ & ... & $X^{n-3}$ & $X^{n-2}$ & $X^{n-1}$\tabularnewline
\cline{1-1} 
$T_{1}^{e}$ & $X^{0}$ & $X^{0}$ & ... & $X^{n-3}$ & $X^{n-2}$ & $X^{n-1}$\tabularnewline
\cline{1-1} 
$G_{1}^{e}$ & $X^{0}$ & $X^{2}$ & ... & $X^{n-3}$ & $X^{n-2}$ & $X^{n-1}$\tabularnewline
\cline{1-1} 
$\vdots$ & $\vdots$ & $\vdots$ & $\vdots$ & $\vdots$ & $\vdots$ & $\vdots$\tabularnewline
\cline{1-1} 
$G_{n-3}^{e}$ & $X^{0}$ & $X^{1}$ & ... & $X^{n-2}$ & $X^{n-2}$ & $X^{n-1}$\tabularnewline
\cline{1-1} 
$T_{n-2}^{e}$ & $X^{0}$ & $X^{1}$ & ... & $X^{n-3}$ & $X^{n-3}$ & $X^{n-1}$\tabularnewline
\cline{1-1} 
$G_{n-2}^{e}$ & $X^{0}$ & $X^{1}$ & ... & $X^{n-3}$ & $X^{n-1}$ & $X^{n-1}$\tabularnewline
\cline{1-1} 
\end{tabular}
\par\end{center}

\end{example}

\section{SM n-Comprehensions with Recursion}

Following \cite{Rom=0000E1n-Par=0000E9}, where some categorical structures
giving rise to primitive recursive functions in the initial monoidal
category with a \emph{left natural numbers object }were introduced\emph{,}
we can establish for some objects in the free \emph{SM} n-Comprehension
with Recursion analogous results. We will see in fact, in the following
Section, that the morphisms generated in the free \emph{SM} n-Comprehension
with Recursion are morphisms between \emph{cocommutative comonoids}
in a \emph{SM} category (see Appendix 2 for a description of these
concepts). This is done to justify the introduction of the safe dependent
recursion schemes in the class of \emph{SM n-Comprehensions with Recursion}
for the so-called\emph{ cartesian objects} below.

We have the following Theorem related to this point (taken from \citep{Baez}). 
\begin{thm}
\label{transmono}Let $\mathcal{C}$ be a \emph{SM} category, $\triangle:\mathcal{C}\longrightarrow\mathcal{C}$
a functor such that $\triangle(C)=C\otimes C$ and $t:\mathcal{C}\longrightarrow\mathcal{C}$
a functor such that $t(C)=\top$ for every object $C$ in $\mathcal{C}$
with monoidal natural transformations $\delta:id\longrightarrow\triangle$
and $\tau:id\longrightarrow t$ such that for every object $C$ in
$\mathcal{C}$ the diagrams {\small{}
\[
\xymatrix{C\ar[r]^{\delta_{C}} & C\otimes C\ar[d]^{C\otimes\tau_{C}}\\
C\ar[u]^{C} & C\otimes\top\ar[l]^{r}
}
\qquad\xymatrix{C\ar[r]^{\delta_{C}} & C\otimes C\ar[d]^{\tau_{C}\otimes C}\\
C\ar[u]^{C} & \top\otimes C\ar[l]^{l}
}
\]
}commute. Then $\mathcal{C}$ is cartesian\emph{ SM}.\end{thm}
\begin{proof}
See Appendix 3.
\end{proof}
This Theorem says essentially that every \emph{SM} category is a cartesian
\emph{SM} category if we can duplicate and delete data and, roughly
speaking, \emph{duplicate and delete the same datum is the same thing
than doing nothing}.%
\footnote{In the original in \citep{Baez} that condition was argued to be actually
necessary and sufficient. We state just a direction for being enough
for the purpose of this paper.%
}

We now define the basic categorical structure from which we'll develop
recursion in \emph{n}-Comprehensions. That is done by taking a class
of SM n-Comprehensions endowed with more structure, that is, some
recursive diagrams. We then proceed to modify and enrich the structure
with initial diagrams and recursive operators. For that we denote
by $\mathcal{CR}^{n}$ a new class named \emph{SM n-Comprehension
with Recursion} obtained from a SM n-Comprehension in the form of
the following Definition.
\begin{defn}
\label{CN}We define the class of \emph{SM} \emph{n}-\emph{Comprehensions
with Recursion}, denoted by $\mathcal{CR}^{n}$, as the class of \emph{SM}
\emph{n}-\emph{Comprehensions} in the form $(\mathcal{C},\left\langle T_{k}\right\rangle ,\left\langle G_{k}\right\rangle ,\left\langle \eta_{k}\right\rangle ,\left\langle \epsilon_{k}\right\rangle )$\end{defn}
\begin{itemize}
\item containing an object $N_{0}$ and two arrows $0_{0}$ and $s_{0}$
whose diagram (named \emph{initial diagram}) is 
\[
\top\overset{0_{0}}{\longrightarrow}N_{0}\overset{s_{0}}{\longrightarrow}N_{0}.
\]
We define recursively for each $i=1,...,n-2$ the objects $N_{i}$
by the rules%
\footnote{$N_{i}$ will be the \emph{levels of $N$}.%
} 
\[
N_{1}=G_{0}N_{0}
\]
\[
N_{i+1}=G_{i}N_{i}
\]
and morphisms $0_{j}$ and $s_{j}$.%
\footnote{Defined by the following schemes: $\begin{cases}
0_{1}=G_{0}(0_{0})\\
0_{j+1}=G_{j}(0_{j})
\end{cases}$ and $\begin{cases}
s_{1}=G_{0}(s_{0})\\
s_{j+1}=G_{j}(s_{j})
\end{cases}$.%
} In $\mathcal{C}$ we have for each $i=0,1,...,n-2$ and $j=0,1,...,n-1$

\[
T_{i}N_{j}=\begin{cases}
\top & \textrm{if \ensuremath{i=j=0}}\\
N_{i-1} & \textrm{if \ensuremath{i=j\neq0}}\\
N_{j} & \textrm{otherwise}
\end{cases}\qquad G_{i}N_{j}=\begin{cases}
N_{i+1} & \textrm{if \ensuremath{i=j}}\\
N_{j} & \textrm{otherwise}
\end{cases}.
\]
With these definitions we can generate all initial diagrams in the
form 
\[
\top\overset{0_{j}}{\longrightarrow}N_{j}\overset{s_{j}}{\longrightarrow}N_{j}
\]
for $0<j<n-1$ as well as 
\[
T_{0}(\top\overset{0}{\longrightarrow}N_{0}\overset{s}{\longrightarrow}N_{0})=\top\overset{1}{\longrightarrow}\top\overset{1}{\longrightarrow}\top;
\]

\item closed under\emph{ flat recursion }(\emph{$FR$}):

for all morphisms 
\[
g:X\longrightarrow Y\textrm{ and }h:N_{0}\otimes X\longrightarrow Y
\]
where $X$ and $Y$ are in the form $N_{0}^{\alpha}$ there exist
a unique 
\[
f:N_{0}\otimes X\longrightarrow Y
\]
in $\mathcal{C}$, which we will denote by $FR(g,h)$, such that the
following diagram commutes%
\footnote{This is actually a coproduct diagram. By applying $G$ to this diagram
we obtain flat recursion for successive levels of $N$, we denote
them by $FR_{k}$ for $1\leq k\leq n-2$. $FR_{k}$ diagrams give
to the initial diagrams appropriate properties such as the injectivity
of the \emph{successor function} $s$.%
}
\[
\xymatrix{\top\otimes X\ar[rr]^{0_{0}\otimes X}\ar[drr]_{g\circ l} &  & N_{0}\otimes X\ar[d]^{f} &  & N_{0}\otimes X\ar[dll]^{h}\ar[ll]_{s_{0}\otimes X}\\
 &  & Y
}
\]

\item closed under\emph{ safe ramified recursion} diagrams on each level
$k$ ($SRR_{k}$):

for all $k=0,1,...,n-2$ and for all morphisms 
\[
g:X\longrightarrow Y\textrm{ and }h:Y\longrightarrow Y
\]
where $T_{k}...T_{0}Y$ is isomorphic to $\top$ there exist a unique
\[
f:N_{k+1}\otimes X\longrightarrow Y
\]
in $\mathcal{C}$, which we will denote by $SRR_{k}(g,h)$, such that
the following diagram commutes 
\[
\xymatrix{\top\otimes X\ar[rr]^{0\otimes X}\ar[d]_{l} &  & N_{k+1}\otimes X\ar[rr]^{s\otimes X}\ar[d]^{f} &  & N_{k+1}\otimes X\ar[d]^{f}\\
X\ar[rr]_{g} &  & Y\ar[rr]_{h} &  & Y
}
\]

\item naming \emph{cartesian objects} in $\mathcal{CR}^{n}$ the objects
in the form {\small{}$\underset{i=0}{\overset{n-1}{\bigotimes}}N_{i}^{\alpha_{i}}$,
}we have that for every cartesian object $\mathcal{CR}^{n}$ is also
closed under \emph{safe dependent recursion} in each level $k$ ($SDR{}_{k}$):

for all $k=0,...,n-2$ and for all morphisms 
\[
g:X\longrightarrow Y\textrm{ and }h:(N_{k+1}\otimes X)\otimes Y\longrightarrow Y
\]
where $T_{k}...T_{0}Y$ is isomorphic to $\top$ and $X$ and $Y$
are cartesian objects there exist a unique 
\[
f:N_{k+1}\otimes X\longrightarrow Y
\]
in $\mathcal{C}$, which we will denote by $SDR_{k}(g,h)$, such that
the following diagram commutes

\noindent \begin{center}
\hspace{10em}\xymatrix{\top\otimes X\ar[rr]^{0_{k+1}\otimes X}\ar[drr]_{\hspace{-1em}(0_{k+1}\otimes X),g\circ\pi_{1}} &  & N_{k+1}\otimes X\ar[rr]^{s_{k+1}\otimes X}\ar[d]^{id,f} &  & N_{k+1}\otimes X\ar[d]^{f}\\ &  & (N_{k+1}\otimes X)\otimes Y\ar[rr]_{h} &  & Y}
\par\end{center}

\end{itemize}
Elements of $\mathcal{CR}^{n}$ are then \emph{SM} \emph{n-Comprehensions}
with four different shaped diagrams and certain bounding conditions
on the objects over which those diagrams are acting. Note at this
point also that the number of nested recursions made in every step
is exactly the recursion level in every scheme (see \citep{Bellantoni-Niggl}).
\begin{example}
\label{ex}Our exemple of SM n-Comprehension with Recursion consists
of defining a cotensor in the form of a presheaf. Consider the category
$Set^{n^{op}}$ which we denote by $\widehat{Set}$. Its objects are
chains of sets indexed by $\mathbf{n}^{op}$: 
\[
X_{n-1}\overset{f_{n-2}}{\longrightarrow}...\overset{f_{0}}{\longrightarrow}X_{0}
\]
and its arrows squares built out of them.

By fixing a single set $X$ we have some special objects $X^{k}$
for $k=0,...,n-1$ in the same form than those given in the Exemple
\ref{setn}\emph{ }
\begin{itemize}
\item $\widehat{Set}$ is a SM category
\item It has as terminal object chains $1\rightarrow\ldots\rightarrow1$
denoted by $1^{n}$ where $1$ is whatever set with a single object
\item For $k\in\{0,1,...,n-1\}$ and taking $0$ (\emph{zero}) and \emph{$s$
}(\emph{successor}) from the usual diagram $1\overset{0}{\longrightarrow}\mathbf{N}\overset{s}{\longrightarrow}\mathbf{N}$
in $Set$ we have the chains of functions

\begin{itemize}
\item $0_{k}:1^{n}\rightarrow\mathbf{N}^{k}$ in the form
\[
\xymatrix{1\ar[r]\ar[d]_{0} & \cdots\ar[r] & 1\ar[d]_{0}\ar[r] & 1\ar[d]\ar[r] & \cdots\ar[r] & 1\ar[d]\\
\mathbf{N}\ar[r] & \cdots\ar[r] & \mathbf{N}\ar[r] & 1\ar[r] & \cdots\ar[r] & 1
}
\]
with $k$ zero arrows and $n-k-1$ arrows with no name which are identities
\item $s_{k}:\mathbf{N}^{k}\rightarrow\mathbf{N}^{k}$ in the form
\[
\xymatrix{\mathbf{N}\ar[r]\ar[d]_{s} & \cdots\ar[r] & \mathbf{N}\ar[d]_{s}\ar[r] & 1\ar[d]\ar[r] & \cdots\ar[r] & 1\ar[d]\\
\mathbf{N}\ar[r] & \cdots\ar[r] & \mathbf{N}\ar[r] & 1\ar[r] & \cdots\ar[r] & 1
}
\]
with $k$ successor arrows and $n-k-1$ arrows with no name which
are identities.
\end{itemize}
\item We define the endofunctors $T_{k}^{\widehat{e}}\textrm{ and }G_{k}^{\widehat{e}}$
in $\widehat{Set}$ in the same way than Exemple \ref{setn} but reversing
the subindexes:

for every $X_{n-1}\overset{h_{n-2}}{\longrightarrow}...\overset{h_{0}}{\longrightarrow}X_{0}$
we obtain 
\[
T_{k}^{\widehat{e}}(X_{n-1}\rightarrow...\rightarrow X_{0})=
\]
\[
=X_{n-1}\rightarrow...\rightarrow X_{n-1-k}\overset{id}{\longrightarrow}X_{n-1-k}\overset{t}{\longrightarrow}X_{n-3-k}\rightarrow...\rightarrow X_{0}
\]
and
\[
G_{k}^{\widehat{e}}(X_{n-1}\rightarrow...\rightarrow X_{0})=
\]
\[
X_{n-1}\rightarrow...\rightarrow X_{n-k}\overset{g}{\longrightarrow}X_{n-2-k}\overset{id}{\longrightarrow}X_{n-2-k}\rightarrow...\rightarrow X_{0}
\]
where $t=h_{n-3-k}\circ h_{n-2-k}$ and $g=h_{n-1-k}\circ h_{n-2-k}$
and for every chain of vertical arrows $(f_{n-1},...,f_{0})$ we obtain%
\footnote{With the notation established in the description of $\mathbf{n}\multimap\mathcal{C}$
in Appendix 1.%
} 
\[
T_{k}^{\widehat{e}}(f_{n-1},...,f_{0})=(f_{n-1},...,f_{n-1-k},f_{n-1-k},f_{n-3-k},...,f_{0})
\]
and
\[
G_{k}^{\widehat{e}}(f_{n-1},...,f_{0})=(f_{n-1},...,f_{n-k},f_{n-k},f_{n-2-k},...,f_{0})
\]

\item We can define a bifunctor $\Im:M_{n}^{op}\rightarrow(\widehat{Set},\widehat{Set})$
sending every $T_{k},G_{k},\eta_{k},\epsilon_{k}$ to $T_{k}^{\widehat{e}},G_{k}^{\widehat{e}},\eta_{k}^{\widehat{e}},\epsilon_{k}^{\widehat{e}}$
respectively.
\end{itemize}
\end{example}
\begin{prop}
\label{cart}Every cartesian object in $\mathcal{C}\in\mathcal{CR}^{n}$
is endowed with \emph{diagonal} and \emph{eraser} \emph{morphisms}
satisfying the hypothesis of Theorem \ref{transmono}. \end{prop}
\begin{proof}
\emph{Eraser} and \emph{duplication} morphisms can be both defined
on every cartesian object in $\mathcal{CR}^{n}$. Let then be $X$
a cartesian object belonging to $\mathcal{C}$ in $\mathcal{CR}^{n}$:
\begin{enumerate}
\item \emph{Eraser} morphisms $\tau_{X}:X\longrightarrow\top$ in $\mathcal{C}$
can be defined recursively by considering:

\begin{itemize}
\item if $X=\top$ we take $\tau_{\top}=1_{\top}$
\item if $X=N_{k+1}$ with $k=0,...,n-1$ we can form the following instance
of safe ramified recursion 
\[
\xymatrix{\top\otimes\top\ar[rr]^{0_{k+1}\otimes\top}\ar[drr]_{l} &  & N_{k+1}\otimes\top\ar[d]^{f}\ar[rr]^{s_{k+1}\otimes\top} &  & N_{k+1}\otimes\top\ar[d]^{f}\\
 &  & \top\ar[rr]_{id} &  & \top
}
\]
and the composition $\tau_{N_{k+1}}=f\circ r^{-1}$%
\footnote{To obtain the arrow $\tau_{N_{0}}:N_{0}\longrightarrow\top$ we take
$\eta N_{0}$.%
} 
\item if $X=Y\otimes Z$ with $Y$ and $Z$ in any of the former cases then
we also have the \emph{eraser} morphism by recalling that $\tau_{X}X=\tau_{Y}Y\otimes\tau_{Z}Z$.
\end{itemize}
\item \emph{Duplication} morphisms $\delta_{N_{k}}$ can be obtained by
the following diagrams 
\[
\xymatrix{\top\otimes\top\ar[rr]^{0_{k+1}\otimes\top}\ar[drr]_{0_{k}\otimes0_{k}} &  & N_{k+1}\otimes\top\ar[rr]^{s_{k+1}\otimes\top}\ar[d]^{f} &  & N_{k+1}\otimes\top\ar[d]^{f}\\
 &  & N_{k}\otimes N_{k}\ar[rr]_{s_{k}\otimes s_{k}} &  & N_{k}\otimes N_{k}
}
\]
for each $k=1,...,n-2$ and the composition $\delta_{N_{k}}=G_{k}(f\circ r^{-1})$.
\end{enumerate}

Squares of Theorem \ref{transmono} involving eraser and duplication
are also commutative in $\mathcal{C}\in\mathcal{CR}^{n}$ because
of their uniqueness.

\end{proof}
\begin{rem}
$\delta_{N_{0}}$ has the problem that we have not at our disposal
neither a diagram giving it nor coercion functors allowing us, when
$n=2$, to lower the level of the object over which it is acting.
We'll consider therefore in the sequel $n>2$.\end{rem}
\begin{example}
Exemple \ref{ex} can be extended to get cartesian objects. They exist
obviously in $\widehat{Set}$ as those chains of sets $X_{n-1}\rightarrow...\rightarrow X_{0}$
where each $X_{k}$ is of the form $\underset{i=0}{\overset{n-1}{\bigotimes}}\mathbf{N}_{i}^{\alpha_{i}}$
for $k=0,...,n-1$.
\end{example}
With the last result we point that every cartesian object\emph{ }in\emph{
}$\mathcal{C}\in\mathcal{CR}^{n}$\emph{ }behave as we expect, that
is, they are really cartesian in the sense of Theorem \ref{transmono}.
That concept of cartesian object was devoted in the Definition of
$\mathcal{CR}^{n}$ to introduce the so-called \emph{safe dependent
recursion} and is inspired on the results in \citep{Rom=0000E1n-Par=0000E9},
where it was proven that all the objects in the initial monoidal category
with a \emph{left natural numbers object} are powers of it.

\section{The free \emph{SM n}-Comprehension with Recursion}

By endowing the initial \emph{SM} category with all \emph{initial
diagrams} and all required recursion schemes, we consider the free
\emph{SM n-Comprehension with Recursion, }which we\emph{ }denote $\mathcal{FR}^{n}$\emph{.
}Now, regarding some concepts of the previous section and some results
of \citep{Rom=0000E1n-Par=0000E9}, we see that $\mathcal{FR}^{n}$
is actually a cartesian \emph{SM} category, which allows us to consider
$SDR$ diagrams in it. 
\begin{thm}
$\mathcal{FR}^{n}$ is cartesian.\end{thm}
\begin{proof}
It is a consequence of Proposition \ref{cart}.
\end{proof}
Now we have the following results related to the concept of cocommutative
comonoid, given in Appendix 2, that were first stated in \citep{Rom=0000E1n-Par=0000E9}.
We won't mention the subscripts of $\delta$ and $\tau$ when they
are obvious and $n$ will be greater than $2$ for the following.
\begin{prop}
$(N_{i},\delta,\tau)$ are cocommutative comonoids in $\mathcal{FR}^{n}$
for all $i=0,1,...,n-1$.\end{prop}
\begin{cor}
$(N_{i}^{k},\delta,\tau)$ are cocommutative comonoids in $\mathcal{FR}^{n}$
for all $k\in\mathbf{N}$ and for all $i=0,1,...,n-1$.\end{cor}
\begin{thm}
The tensor product of two cartesian objects in $\mathcal{FR}^{n}$
is a cartesian product.\end{thm}
\begin{proof}
All cartesian objects in $\mathcal{FR}^{n}$ are cocommutative comonoids.
\end{proof}
It's important here to note that this Theorem allowed us to introduce
$SDR$ diagrams in $\mathcal{FR}^{n}$ as it was seen in Proposition
\ref{cart}.

\section{The standard model}

The\emph{ Freyd} \emph{Cover}, technique that we will use to prove
some properties of the syntactical structures defined up to now, is
a particular case of the following Definition.
\begin{defn}
Given a functor $\Gamma:\mathcal{C}\longrightarrow Set$ we call \emph{Artin
Glueing} the comma category $\nicefrac{Set}{\Gamma}$ generated from
$\Gamma$: 
\begin{itemize}
\item whose objects are groups of three $(X,f,U)$ where 

\begin{itemize}
\item $X$ is a set
\item $U$ is an an object of $\mathcal{C}$
\item $f$ is a function $X\longrightarrow\Gamma U$
\end{itemize}
\item whose morphisms between the objects $(X,f_{1},U)$ and $(Y,f_{2},V)$
are commutative squares

\[
\xymatrix{X\ar[r]^{h_{1}}\ar[d]_{f_{1}} & Y\ar[d]^{f_{2}}\\
\Gamma U\ar[r]_{\Gamma h_{2}} & \Gamma V
}
\]
that is, ordered pairs $(h_{1},h_{2})$ where $X\overset{h_{1}}{\longrightarrow}Y$
and $U\overset{h_{2}}{\longrightarrow}V$.

\end{itemize}
\begin{defn}
If $\mathcal{C}$ is a category with a terminal object $1$ its \emph{Freyd
Cover} is the Artin Glueing for the functor $\Gamma=\mathcal{C}(1,-)$.
\end{defn}
\end{defn}
Morphisms in $\mathcal{FR}^{n}$ that we will call \emph{formal},
because of their resemblance with the terms in the formal languages,
can be identified with programs generated in that category. 
\begin{defn}
\label{gamma}The\emph{ standard model of formal morphisms} is the
functor $\Gamma_{n}$ given by the diagram
\[
\xymatrix{\mathcal{FR}^{n}\ar[rr]^{\Gamma_{n}}\ar[dr]_{\overline{\chi}} &  & n\multimap Set\\
 & n\multimap\mathcal{FR}^{n}\ar[ur]_{n\multimap\Gamma}
}
\]
that is $\Gamma_{n}=(n\multimap\Gamma)\circ\overline{\chi}$ where
$\Gamma:\mathcal{FR}^{n}\longrightarrow Set$ is defined by $\Gamma X=\mathcal{F}^{n}(\top,X)$
and $\Gamma f=f\circ-$.%
\footnote{This is a special case of the \emph{global sections functor}.%
}

Taking into account that the functor $\Gamma_{n}$ acts over the objects
$\top$ and $N_{n-1}$ in $\mathcal{FR}^{n}$ as
\[
\Gamma_{n}\top=1\longrightarrow1\longrightarrow\ldots\longrightarrow1
\]
and 
\[
\Gamma_{n}N_{n-1}=\mathbf{N}_{n-1}\longrightarrow\mathbf{N}_{n-1}\longrightarrow\ldots\longrightarrow\mathbf{N}_{n-1}
\]
where the arrows are identities, its expressions over the elements
in $\mathcal{FR}^{n}$ are:\end{defn}
\begin{itemize}
\item over the objects $N_{j}$ in $\mathcal{\mathcal{FR}}^{n}$ for $0\leq j\leq n-2$
we have 
\[
\Gamma_{n}N_{j}=[(n\multimap\Gamma)\circ\overline{\chi}](N_{j})
\]
being equal to $n\multimap\Gamma$ applied to%
\footnote{We point here that we have the following identities:
\[
\overline{k}N_{j}=\begin{cases}
\top & \textrm{ if }0\leq j\leq k-1\\
N_{n-1} & \textrm{ other }
\end{cases}\qquad\overline{k}\mathbf{N}_{j}=\begin{cases}
1 & \textrm{ if }0\leq j\leq k-1\\
\mathbf{N}_{n-1} & \textrm{\textrm{ other }}
\end{cases}
\]
}
\[
\xymatrix{\overline{0}N_{j}\ar[r] & \overline{1}N_{j}\ar[r] &  & \cdots & \textrm{ }\ar[r] & \overline{n-2}N_{j}\textrm{ }\ar[r] & \overline{n-1}N_{j}\textrm{ }}
\]
and giving
\[
\xymatrix{\overline{0}\mathbf{N}_{j}\ar[r] & \overline{1}\mathbf{N}_{j}\ar[r] &  & \cdots & \textrm{ }\ar[r] & \overline{n-2}\mathbf{N}_{j}\textrm{ }\ar[r] & \overline{n-1}\mathbf{N}_{j}\textrm{ }}
\]
This is both an object in $n\multimap Set$ and a function composition
in $Set$.%
\footnote{In terms of sequences out of $1$ and $\mathbf{N}$ we had $n-1$
chains of commutative squares.%
}
\item over morphisms $f:N_{k}\longrightarrow N_{j}$ in $\mathcal{\mathcal{FR}}^{n}$
for $1\leq k,j\leq n-2$ it is represented by commutative squares
in the form 
\[
\Gamma_{n}f=[(n\multimap\Gamma)\circ\overline{\chi}](f)=(n\multimap\Gamma)(\overline{\chi}N_{k}\longrightarrow\overline{\chi}N_{j})
\]
giving {\footnotesize{}
\[
\xymatrix{\overline{0}\mathbf{N}_{k}\ar[dd]_{\overline{0}\chi_{0}\mathbf{N}_{k}}\ar[rrr]^{\overline{0}f} &  &  & \overline{0}\mathbf{N}_{j}\ar[dd]^{\overline{0}\chi_{0}\mathbf{N}_{j}}\\
\\
\overline{1}\mathbf{N}_{k}\ar[dd]_{\overline{1}\chi_{1}\mathbf{N}_{k}}\ar[rrr]^{\overline{1}f} &  &  & \overline{1}\mathbf{N}_{j}\ar[dd]^{\overline{1}\chi_{1}\mathbf{N}_{j}}\\
\\
\textrm{ } &  &  & \textrm{ }\\
\vdots &  &  & \vdots\\
\textrm{ }\ar[d] &  &  & \textrm{ }\ar[d]\\
\overline{n-2}\mathbf{N}_{k}\ar[dd]_{\overline{n-2}\chi_{n-2}\mathbf{N}_{k}}\ar[rrr]^{\overline{n-2}f} &  &  & \overline{n-2}\mathbf{N}_{j}\ar[dd]^{\overline{n-2}\chi_{n-2}\mathbf{N}_{j}}\\
\\
\overline{n-1}\mathbf{N}_{k}\ar[rrr]^{\overline{n-1}f} &  &  & \overline{n-1}\mathbf{N}_{j}
}
\]
}This is both an arrow in $n\multimap Set$ and a chain of commutative
squares in $Set$.%
\footnote{In terms of sequences out of $1$ and $\mathbf{N}$ we had $n-1$
chains of commutative cubes.%
}
\end{itemize}
In a more general case we could consider objects in the form $N_{j}^{\alpha_{j}}$.
In that case, given that the endofunctors $T$ and $G$ preserve the
tensor product, we had chains in the form
\[
\overline{0}\mathbf{N}_{j}\otimes\overset{\alpha_{j}}{\cdots}\otimes\overline{0}\mathbf{N}_{j}\rightarrow...\rightarrow\overline{n-1}\mathbf{N}_{j}\otimes\overset{\alpha_{j}}{\cdots}\otimes\overline{n-1}\mathbf{N}_{j}
\]
and an analogous expression for the morphisms. We will work \emph{modulo
tensor powers} due to the enormous length of those expressions.
\begin{defn}
In the case of the $n$-Comprehension $n\multimap Set$ the Freyd
Cover of $\mathcal{\mathcal{FR}}^{n}$ is given by the comma category
$\nicefrac{(n\multimap Set)}{\Gamma_{n}}$ with the functor $\Gamma_{n}:\mathcal{\mathcal{FR}}^{n}\longrightarrow n\multimap Set$
whose 
\begin{itemize}
\item objects are triples $(X,f,U)$ where

\begin{itemize}
\item $X$ is an object of $n\multimap Set$, that is, a chain in the form
\[
X_{0}\rightarrow X_{1}\rightarrow...\rightarrow X_{n-1}
\]

\item $U$ is an object of $\mathcal{\mathcal{FR}}^{n}$, that is, the tensor
product of distinct tensor powers of objects $N_{k}$ in the form
\[
\overset{n-1}{\underset{j=0}{\bigotimes}}N_{j}^{\alpha_{j}}
\]
 
\item $f$ is a function $X\longrightarrow\Gamma_{n}U$ in $n\multimap Set$,
that is, a chain of squares
\end{itemize}
\item morphisms between objects $(X,f_{1},U)$ and $(Y,f_{2},V)$ are commutative
squares

\[
\xymatrix{X\ar[r]^{h_{1}}\ar[d]_{f_{1}} & Y\ar[d]^{f_{2}}\\
\Gamma_{n}U\ar[r]_{\Gamma_{n}h_{2}} & \Gamma_{n}V
}
\]
that is, pairs $(h_{1},h_{2})$ where $X\overset{h_{1}}{\longrightarrow}Y$
belongs to $n\multimap Set$ and $U\overset{h_{2}}{\longrightarrow}V$
to $\mathcal{\mathcal{FR}}^{n}$ and therefore $\xymatrix{\Gamma_{n}U\ar[r]^{\Gamma_{n}h_{2}} & \Gamma_{n}V}
$ also belongs to $n\multimap Set$.

Those squares can be seen as chains of commutative cubes in $n\multimap Set$.

\end{itemize}
\end{defn}
To complete this Section we give two results connecting the syntactical
structure here described with the semantics of numerical functions. 
\begin{prop}
The image of the objects $N_{k}$ by the functor $\Gamma$ are sets
whose elements have the form $\Gamma N_{k}=\{std_{k}n/n\in\mathbf{N}\}$
where $std_{k}:\mathbf{N}\longrightarrow\Gamma N_{k}$ is defined
by the scheme 
\[
\begin{cases}
std_{k}0=0_{k}\\
std_{k}(sn)=s_{k}(std_{k}n)
\end{cases}
\]
with $k=0,1,...,n-1$. \end{prop}
\begin{cor}
$\Gamma N_{k}=\mathbf{N}_{k}$ for all $k=0,1,...,n-1$.
\end{cor}
This Proposition and its Corollary indicate that the sets generated
by the functor $\Gamma$ applied to the levels of the natural numbers
in $\mathcal{\mathcal{FR}}^{n}$ behave as the natural numbers themselves.
This fact is a consequence of the use of the Freyd Cover, where every
arrow $\top\longrightarrow N_{k}$ has the form $s_{k}^{n}\circ0_{k}$
for some $n\in\mathbf{N}$.

\section{Recursive functions in $\mathcal{FR}^{n}$}

To show how hierarchies of subrecursive functions can be defined in
$\mathcal{FR}^{n}$ we introduce a language containing \emph{n} different\emph{
species of variables}, separated by semicolons, which we will denote
by the numbers $0,1,...,n-1$. We assign at the same time a \emph{level}
to every function as is explained in the following:
\begin{itemize}
\item we say that a function $f$ is \emph{of the type} $(a_{k},a_{k-1},...,a_{0};a_{m})$
if its arguments belong to the species $a_{k},a_{k-1},...,a_{0}$
in its domain and its codomain belong to the species $a_{m}$. We
express this fact by 
\[
\underset{a_{k}a_{k-1}...a_{0};a_{m}}{f}
\]

\item we define the \emph{level of a function} as the species of its codomain.
\end{itemize}
If a variable belongs to the \emph{n}-th species then it also belongs
to the \emph{(n+1)}-th species.

Now we need to make use of a new recursion scheme in $\mathcal{\mathcal{F}R}^{n}$
with \emph{$n>2$ }that will turn out to be a particular instance
of $SDR$ scheme.
\begin{defn}
\noindent We say that a morphism $f:N_{k+1}\otimes X\longrightarrow Y$
in $\mathcal{\mathcal{FR}}^{n}$ with \emph{$n>2$} is defined by
the \emph{parameterized safe ramified recursion scheme} on the level
$k$ if it is the unique such that for all $g:X\longrightarrow Y$
and $h:X\otimes Y\longrightarrow Y$ with $T_{k}...T_{0}Y$ isomorphic
to $\top$ the following diagram
\[
\xymatrix{\top\otimes X\ar[rr]^{0_{k+1}\otimes X}\ar[d]_{l} &  & N_{k+1}\otimes X\ar[rr]^{s_{k+1}\otimes X}\ar[d]^{\pi_{1},f} &  & N_{k+1}\otimes X\ar[d]^{f}\\
X\ar[rr]_{id,g} &  & X\otimes Y\ar[rr]_{h} &  & Y
}
\]
commutes. We denote $f$ by $PSRR{}_{k}(g,h)$.\end{defn}
\begin{thm}
Every function defined using a $PSRR$ scheme can also be defined
using a $SDR$ scheme.
\end{thm}
With this result we can argue that a doctrine which is closed under
the $SDR$ scheme is also closed under the $PSRR$ scheme.

We now define some functions:%
\footnote{They are functions belonging to the so-called \emph{Hyperoperation
Sequence}, which gives an easy way to classify the functions into
the \emph{Grzegorzcyk} \emph{Hierarchy }by its complexity. %
}
\begin{itemize}
\item \emph{addition }in $\mathcal{\mathcal{FR}}^{1}$ denoted by $\underset{10;0}{\bigoplus}:N_{1}\otimes N_{0}\longrightarrow N_{0}$
is defined by $SRR$: {\small{}
\[
\xymatrix{\top\otimes N_{0}\ar[rr]^{0_{1}\otimes N_{0}}\ar[d]_{l} &  & N_{1}\otimes N_{0}\ar[rr]^{s_{1}\otimes N_{0}}\ar[d]^{\bigoplus} &  & N_{1}\otimes N_{0}\ar[d]^{\bigoplus}\\
N_{0}\ar[rr]_{id} &  & N_{0}\ar[rr]_{s_{0}} &  & N_{0}
}
\]
}such that {\small{}
\[
\begin{cases}
\bigoplus(0,n)=n\\
\bigoplus(sx,n)=s(\bigoplus(x,n))
\end{cases}
\]
}{\small \par}
\item \emph{multiplication} in $\mathcal{\mathcal{F}R}^{2}$ denoted by
$\underset{11;0}{\bigotimes}:N_{1}\otimes N_{1}\longrightarrow N_{0}$
is defined by $PSRR$: {\small{}
\[
\xymatrix{\top\otimes N_{1}\ar[rr]^{0_{1}\otimes N_{1}}\ar[d]_{l} &  & N_{1}\otimes N_{1}\ar[rr]^{s_{1}\otimes N_{1}}\ar[d]^{\pi_{1},\bigotimes} &  & N_{1}\otimes N_{1}\ar[d]^{\bigotimes}\\
N_{1}\ar[rr]_{id,0\circ\tau_{N_{1}}} &  & N_{1}\otimes N_{0}\ar[rr]_{\bigoplus} &  & N_{0}
}
\]
such that} {\small{}
\[
\begin{cases}
\bigotimes(0,y)=0\\
\bigotimes(sx,y)=\bigoplus(y,\bigotimes(x,y))
\end{cases}
\]
} 
\item \emph{exponentiation} in $\mathcal{\mathcal{FR}}^{3}$ denoted by
$\underset{21;1}{\uparrow}:N_{2}\otimes N_{1}\longrightarrow N_{1}$
is defined by $PSRR$:%
\footnote{$c_{1}$ is the constant function $1$.%
} {\small{}
\[
\xymatrix{\top\otimes N_{1}\ar[rr]^{0_{2}\otimes N_{1}}\ar[d]_{l} &  & N_{2}\otimes N_{1}\ar[rr]^{s_{2}\otimes N_{1}}\ar[d]^{\pi_{1},\uparrow} &  & N_{2}\otimes N_{1}\ar[d]^{\uparrow}\\
N_{1}\ar[rr]_{id,c_{1}} &  & N_{1}\otimes N_{1}\ar[rr]_{G_{0}\bigotimes} &  & N_{1}
}
\]
}such that {\small{}
\[
\begin{cases}
\uparrow(0,y)=c_{1}\\
\uparrow(sx,y)=G_{0}\bigotimes(y,\uparrow(x,y))
\end{cases}
\]
} 
\item \noindent \emph{tetration }in $\mathcal{\mathcal{FR}}^{4}$ denoted
by $\underset{32;1}{\upuparrows}:N_{3}\otimes N_{2}\longrightarrow N_{1}$
is defined by $PSRR$: {\small{}
\[
\xymatrix{\top\otimes N_{2}\ar[rr]^{0_{3}\otimes N_{2}}\ar[d]_{\pi_{1}} &  & N_{3}\otimes N_{2}\ar[rr]^{s_{3}\otimes N_{2}}\ar[d]^{\pi_{1},\upuparrows} &  & N_{3}\otimes N_{2}\ar[d]^{\upuparrows}\\
N_{2}\ar[rr]_{id,\eta_{2}N_{2}} &  & N_{2}\otimes N_{1}\ar[rr]_{\uparrow} &  & N_{1}
}
\]
}such that {\small{}
\[
\begin{cases}
\upuparrows(0,y)=y\\
\upuparrows(sx,y)=\uparrow(y,\upuparrows(x,y))
\end{cases}
\]
}{\small \par}
\end{itemize}

\section{Safe composition}

Safe composition, as defined in the following Definition, has a representation
in $\mathcal{FR}^{n}$ by means of diagrams associated to natural
transformations in the form $T_{0}...T_{k-1}\eta_{k}$.
\begin{defn}
\label{bn}We say that a function $f$ is defined by \emph{safe composition}
from functions $\overline{r_{0}},...,\overline{r_{n}}$ and $h$ if
\[
f(\overline{x_{n}};...;\overline{x_{0}})=h(\overline{r_{n}}(\overline{x_{n}};);\overline{r_{n-1}}(\overline{x_{n}};\overline{x_{n-1}});...;\overline{r_{0}}(\overline{x_{n}};...;\overline{x_{0}}))
\]
where the level of $f$ is the level of $h$ while the level of $r_{n}$
is less or equal than $n$ and that of $r_{0}$ is $0$.
\end{defn}
For every $\eta_{k}$ and $f:\overset{n-1}{\underset{j=0}{\bigotimes}}N_{j}^{\alpha_{j}}\longrightarrow N_{m}^{\beta}$
morphism in\emph{ }$\mathcal{FR}^{n}$ we have commutative diagrams
in the following form:

\hspace{2em}\xymatrix{T_{0}...T_{k-1}(\overset{n-1}{\underset{j=k}{\bigotimes}}N_{j}^{\alpha_{j}})\ar[rr]^{T_{0}...T_{k-1}f}\ar[dd]_{T_{0}...T_{k-1}\eta_{k}(\overset{n-1}{\underset{j=k}{\bigotimes}}N_{j}^{\alpha_{j}})} &  & T_{0}...T_{k-1}N_{m}^{\beta}\ar[dd]^{T_{0}...T_{k-1}\eta_{k}N_{m}^{\beta}}\\\\T_{0}...T_{k}(\overset{n-1}{\underset{j=k}{\bigotimes}}N_{j}^{\alpha_{j}})\ar[rr]_{T_{0}...T_{k}f} &  & T_{0}...T_{k}N_{m}^{\beta}}

\noindent obtained by the action of $T_{0}...T_{k-1}\eta_{k}$ with
$k=0,1,...,n-2$ over $f$. 

In this diagram we have made use of the identities%
\footnote{Working up to isomorphisms $l$ and $r$.%
} {\small{}
\[
T_{0}...T_{k}(\overset{n-1}{\underset{j=0}{\bigotimes}}N_{j}^{\alpha_{j}})=\overset{n-1}{\underset{j=k+1}{\bigotimes}}N_{j}^{\alpha_{j}}
\]
}and the fact that the arrow {\small{}
\[
\eta_{k}(\overset{n-1}{\underset{j=k}{\bigotimes}}N_{j}^{\alpha_{j}}):\overset{n-1}{\underset{j=k}{\bigotimes}}N_{j}^{\alpha_{j}}\longrightarrow T_{k}(\overset{n-1}{\underset{j=k}{\bigotimes}}N_{j}^{\alpha_{j}})
\]
}is actually an arrow{\small{}
\[
\overset{n-1}{\underset{j=k}{\bigotimes}}N_{j}^{\alpha_{j}}\longrightarrow\overset{n-1}{\underset{j=k+1}{\bigotimes}}N_{j}^{\alpha_{j}}
\]
}for every $k=0,1,...,n-3$ with which we have an expression of $f$
in terms of coercions $T_{k}$ due to the fact that they don't change
anything over an object in the form $N_{m}^{\beta}$ for $k\leq m-1$. 

This grabs the formulation of safe composition from Definition \ref{bn}
because we obtain an expression of each morphism in $\mathcal{FR}^{n}$
in terms of other morphisms whose variables belong, as maximum, to
the same species of the former. Therefore, the level $n-1$ output
does not depend on lower species inputs when we are in $\mathcal{FR}^{n}$.
In general, a $s$ species output does not depend on lower species
inputs than $s$.
\begin{thm}
\label{nnn}For every function 
\[
h(\overline{x}_{n};...;\overline{x_{k+1}};z,\overline{x_{k}};...;\overline{x_{0}})
\]
where $0\leq k<n$ there exists a function 
\[
f(\overline{x}_{n};...;\overline{x_{k+1}},z;\overline{x_{k}};...;\overline{x_{0}})
\]
obtained by safe composition from $h$ and projections such that 
\[
h(\overline{x}_{n};...;\overline{x_{k+1}};z,\overline{x_{k}};...;\overline{x_{0}})=f(\overline{x}_{n};...;\overline{x_{k+1}},z;\overline{x_{k}};...;\overline{x_{0}})
\]
\end{thm}
\begin{proof}
Take projection functions as $\overline{r}$.
\end{proof}
That is, every variable being in a species $k$ position can be moved
to a species $t>k$ position.

The function classes characterized by this setting will satisfy one
of the main features of the subrecursive hierarchies,\emph{ }that
is, its growing behaviour: there exist functions not belonging to
any previous class in their ordering. Take for exemple those of the
\emph{Hyperoperation Sequence} and its relation with the classes in
the \emph{Grzegorzcyk Hierarchy} denoted by $\mathcal{E}^{n}$ for
$n\in\mathbf{N}$. Every \emph{(n+1)}-level function in the \emph{Hyperoperation
Sequence} belong to $\mathcal{E}^{n+1}$ but not to $\mathcal{E}^{n}$. 

Concurrently, we can give in $\mathcal{E}^{k}$ a copy of each function
in $\mathcal{E}^{j}$ for every $k\geq j$ and we forbid in $\mathcal{E}^{j}$
any copy of a function generated in $\mathcal{E}^{k}$. The former
is done by the action of a coercion functor $G_{m}$ for $k>m\geq j$
and the latter by avoiding the application of endofunctors $T_{m}$
for $k\geq m>j$ over the arrows generated by means of a recursion
scheme in $\mathcal{E}^{k}$. This is done to avoid the structure
collapse due to the fact that those coercion functors may reduce subindexes.
In these situations we must consider a subcategory $\mathcal{SFR}^{n}$
of $\mathcal{FR}^{n}$ which we describe in Appendix 4.

\section{Conclusions and future work}

Symmetric Monoidal \emph{n-Comprehensions} are proved to be useful
for new characterizations of subrecursive function classes, giving
a wider point of view of recursion in Category Theory. 

This work can be extended, for instance, by considering other (partial)
orders as giving rise to a different concept of \emph{n}-\emph{Comprehension}
to chase different function classes (see \citep{Otto} for this particular).
Other investigation line to follow starting from this paper could
be a fibrational point of view of the results here given.

\section*{Appendix 1}

We will use some concepts of \citep{Kelly} to get the cotensor of
two $\mathcal{V}$-categories. For that \emph{$\mathcal{V}$ }will
be in the sequel a monoidal category.
\begin{defn}
Let $\mathcal{B}$ be a $\mathcal{V}-$category, $B,C\in\mathcal{B}$
and $X\in\mathcal{V}$. If there exists an object $D$ in $\mathcal{B}$
and a $\mathcal{V}$-natural isomorphism 
\[
\mathcal{B}(B,D)\cong[X,\mathcal{B}(B,C)]
\]
we say that $D$ is the \emph{cotensor product of $X$ and $C$ in
$\mathcal{B}$ }and we will denote it by\emph{ $X\multimap C$.} 

If it exists for all $X$ and $C$ then we say that the $\mathcal{V}-$category
$\mathcal{B}$ is \emph{cotensorial}. In this case we also have the
isomorphism $\mathcal{B}(X\otimes B,C)\cong[X,\mathcal{B}(B,C)]$
which means that we have the isomorphism 
\[
\mathcal{B}(X\otimes B,C)\cong\mathcal{B}(B,X\multimap C)
\]
whenever $\mathcal{V}$ is \emph{SM }closed and its \emph{underlying}
$\mathcal{V}_{0}$ is complete. \end{defn}
\begin{rem}
\label{SMC}Every \emph{SM} closed category has tensor and cotensor
products for every pair of objects and the cotensor product is the
hom-object formed by those objects.\end{rem}
\begin{example}
Let $\mathcal{C}$ be a\emph{ }SM category. We take $\mathcal{B}=\mathcal{SM},\mathcal{V}=Cat$
and the 2-functors 
\[
G:\mathcal{I}\longrightarrow\mathcal{SM}\textrm{ and }F:\mathcal{I}\longrightarrow Cat
\]
where $\mathcal{I}$ is the unit $\mathcal{V}$-category such that
$G$ determines a category $\mathcal{D}$ and $F$ determines $\mathbf{n}$.
Then for all $\mathcal{C}\in\mathcal{SM}$ we have

{\small{}
\[
\mathcal{SM}(\mathbf{n}\otimes\mathcal{D},\mathcal{C})\cong\mathcal{SM}(\mathcal{D},\mathbf{n}\multimap\mathcal{C})\cong[\mathbf{n},\mathcal{SM}(\mathcal{D},\mathcal{C})]
\]
}and, by taking $\mathcal{D}=\mathcal{C}$,{\small{}
\[
\mathcal{SM}(\mathbf{n}\otimes\mathcal{C},\mathcal{C})\cong\mathcal{SM}(\mathcal{C},\mathbf{n}\multimap\mathcal{C})\cong[\mathbf{n},\mathcal{SM}(\mathcal{C},\mathcal{C})]
\]
}where the 2-category at right is isomorphic to $\mathcal{SM}(\mathcal{C},\mathcal{C})^{n}$. 

This construction makes sense due to the fact that $\mathcal{SM}$
admits cotensor objects with the category $\mathbf{n}$ whenever $\mathcal{C}\in\mathcal{SM}$,
a fact that we spell out immediately below. $\mathcal{SM}$ can be
seen itself as a $\mathcal{V}$-category with $\mathcal{V}=Cat$ a
\emph{SM }closed category, we can then say that the cotensor object
is exactly the hom-object. That is
\[
\mathbf{n}\multimap\mathcal{C}=[\mathbf{n},\mathcal{C}]
\]

It can be defined a symmetric monoidal structure for $\mathbf{n}\multimap\mathcal{C}$
when $\mathcal{C}$ is in $\mathcal{SM}$ given by the following:
\begin{itemize}
\item as unit we take the following chain of $n-1$ morphisms
\[
\top\longrightarrow\top\longrightarrow\ldots\longrightarrow\top
\]

\item tensor product of the objects

\hspace{7em}\xymatrix{Y_{0}\ar[r]^{y_{0}} &  & \cdots & \ar[r]^(.4){y_{n-2}} & Y_{n-1}}

and

\hspace{7em}\xymatrix{X_{0}\ar[r]^{x_{0}} &  & \cdots & \ar[r]^(.3){x_{n-2}} & X_{n-1}}

is defined by

\hspace{2em}\xymatrix{Y_{0}\otimes X_{0}\ar[rr]^(.6){y_{0}\otimes x_{0}} &  &  & \cdots & \ar[rr]^(.3){y_{n-2}\otimes x_{n-2}} &  & Y_{n-1}\otimes X_{n-1}}

\item tensor product of an object

\hspace{7em}\xymatrix{Y_{0}\ar[r]^{y_{0}} &  & \cdots & \ar[r]^(.4){y_{n-2}} & Y_{n-1}}

and an arrow%
\footnote{We will express squares like this simply as $(f_{0},...,f_{n-1})$.%
}

\hspace{6em}\xymatrix{X_{0}\ar[r]^{x_{0}}\ar[d]_{f_{0}} &  & \cdots & \ar[r]^(.3){x_{n-2}} & X_{n-1}\ar[d]^{f_{n-1}}\\X'_{0}\ar[r]^{x'_{0}} &  & \cdots & \ar[r]^(.3){x'_{n-2}} & X'_{n-1}}

is defined by

\hspace{1em}\xymatrix{Y_{0}\otimes X_{0}\ar[rr]^(.6){y_{0}\otimes x_{0}}\ar[d]_{Y_{0}\otimes f_{0}} &  &  & \cdots & \ar[rr]^(.3){y_{n-2}\otimes x_{n-2}} &  & Y_{n-1}\otimes X_{n-1}\ar[d]^{Y_{n-1}\otimes f_{n-1}}\\Y_{0}\otimes X'_{0}\ar[rr]^(.6){y_{0}\otimes x'_{0}} &  &  & \cdots & \ar[rr]^(.3){y_{n-2}\otimes x'_{n-2}} &  & Y_{n-1}\otimes X'_{n-1}}

\item symmetries are

\hspace{1em}\xymatrix{Y_{0}\otimes X_{0}\ar[rr]^(.6){y_{0}\otimes x_{0}}\ar[d]_{s_{X_{0},Y_{0}}} &  &  & \cdots & \ar[rr]^(.3){y_{n-2}\otimes x_{n-2}} &  & Y_{n-1}\otimes X_{n-1}\ar[d]^{s_{X_{n-1},Y_{n-1}}}\\X_{0}\otimes Y_{0}\ar[rr]^(.6){x_{0}\otimes y_{0}} &  &  & \cdots & \ar[rr]^(.3){x_{n-2}\otimes y_{n-2}} &  & X_{n-1}\otimes Y_{n-1}}

\end{itemize}

\noindent and the rest of diagrams giving the \emph{SM} structure
where commutativity is satisfied in every diagram because of the \emph{SM}
structure in $\mathcal{C}$.

\end{example}

\section*{Appendix 2}
\begin{defn}
Let $\mathcal{C}$ be a \emph{SM} category. We denote by $CC(\mathcal{C})$
the category whose 
\begin{itemize}
\item objects are cocommutative comonoids in $\mathcal{C}$ in the form
$(A,\delta_{A},\tau_{A})$
\item morphisms between cocommutative comonoids $(A,\delta_{A},\tau_{A})$
and $(B,\delta_{B},\tau_{B})$ are morphisms $f:A\longrightarrow B$
in \emph{$\mathcal{C}$} such that the following diagrams commute{\small{}
\[
\xymatrix{A\ar[r]^{f}\ar[d]_{\delta_{A}} & B\ar[d]^{\delta_{B}}\\
A\otimes A\ar[r]_{f\otimes f} & B\otimes B
}
\quad\xymatrix{A\ar[rd]_{\tau_{A}}\ar[rr]^{f} &  & B\ar[dl]^{\tau_{B}}\\
 & \top
}
\]
}{\small \par}
\end{itemize}
\end{defn}
\begin{rem}
$CC(\mathcal{C})$ is cartesian (see \citep{Fox}). That cartesian
product in $CC(\mathcal{C})$ is given by the comonoid $(A\otimes B,\delta_{A\otimes B},\tau_{A\otimes B})$
for $(A,\delta_{A},\tau_{A})$ and $(B,\delta_{B},\tau_{B})$ in $CC(\mathcal{C})$
due to the fact that the following diagram commutes for all $f:C\longrightarrow A$
and $g:C\longrightarrow B$ in $CC(\mathcal{C})${\small{}
\[
\xymatrix{ &  & C\ar[d]^{(f\otimes g)\circ\delta_{C}}\ar[lld]_{f}\ar[rrd]^{g}\\
A &  & A\otimes B\ar[ll]^{r\circ(A\otimes\tau_{B})}\ar[rr]_{l\circ(\tau_{B}\otimes B)} &  & B
}
\]
}where the definition of \emph{projections in a cartesian }SM\emph{
category} (see the following) are given implicitly.\end{rem}
\begin{defn}
A \emph{cartesian symmetric monoidal category} (a\emph{ cartesian
SM category} in the sequel) is a symmetric monoidal category whose
monoidal structure is given by a cartesian product.\end{defn}
\begin{rem}
From this Definition we can argue that the unit of the tensor in the
case of a cartesian \emph{SM} category is a terminal object in the
category.
\end{rem}
Every cartesian \emph{SM} category is endowed with morphisms \emph{diagonal}
in the form $\delta_{C}:C\rightarrow C\otimes C$ and \emph{eraser}
in the form $\tau_{C}:C\rightarrow\top$ for every object $C$. We
can think on the interpretation of morphisms $\delta$ and $\tau$
in terms of Computer Science as \emph{the one that duplicates a datum}
and \emph{the one that deletes a datum} respectively. Those morphisms
carry the structure of a \emph{cocommutative comonoid} over an object
in the category. In fact, every object in a cartesian \emph{SM} category
can be seen uniquely as a comonoid as seen in the following.
\begin{thm}
Given a cartesian symmetric monoidal category $\mathcal{C}$ every
object is endowed with a cocommutative comonoid structure uniquely
defined.\end{thm}
\begin{proof}
By being a cartesian symmetric monoidal category we know that the
unit $\top$ is a terminal object and therefore there exists for every
object $C$ in $\mathcal{C}$ a unique arrow $C\longrightarrow\top$
that has to be $\tau_{C}$ for the comonoid structure.

On the other hand, by being cartesian we can construct a commutative
diagram in the form
\[
\xymatrix{ & C\ar[d]^{h}\ar[rd]^{id}\ar[ld]_{id}\\
C & C\otimes C\ar[r]_{\pi_{2}}\ar[l]^{\pi_{1}} & C
}
\]
where the unique $h$, denoted by $\left\langle id,id\right\rangle $,
has to be the duplication arrow $\delta_{C}$.
\end{proof}

\section*{Appendix 3}
\begin{proof}
{[}of Theorem \ref{transmono}{]} For an object $D$ in $\mathcal{C}$
and arrows $f_{1}:D\longrightarrow C_{1}$ and $f_{2}:D\longrightarrow C_{2}$
we can construct a diagram%
\footnote{And an analogous for $C_{2}$.%
} {\small{}
\[
\qquad\xymatrix{D\ar[d]_{f_{1}}\ar[r]^{\delta_{D}}\ar@{}[dr]|{\mathstrut\raisebox{2ex}{\ensuremath{(1)}\kern0.5em }} & D\otimes D\ar[r]^{f_{1}\otimes f_{2}}\ar[d]^{f_{1}\otimes f_{1}}\ar@{}[dr]|{\mathstrut\raisebox{2ex}{\textrm{ }\textrm{ }\ensuremath{(2)}\kern0.3em }} & C_{1}\otimes C_{2}\ar[d]^{C_{1}\otimes\tau_{C_{2}}}\\
C_{1}\ar[r]_{\delta_{C_{1}}} & C_{1}\otimes C_{1}\ar[r]_{C_{1}\otimes\tau_{C_{1}}}\ar@{}[d]|{\mathstrut\raisebox{2ex}{\textrm{ }\textrm{ }\ensuremath{(3)}\kern0.3em }} & C_{1}\otimes\top\ar@/^{2pc}/[ll]^{r}\\
 & \textrm{ }
}
\]
}{\small \par}

\noindent where $(1)$ commutes for being $\delta$ a natural transformation
and $(3)$ by hypothesis while the commutativity of diagram numbered
$(2)$ can be proved by considering it as {\small{}
\[
\xymatrix{D\otimes D\ar[r]^{D\otimes f_{2}}\ar[d]_{f_{1}\otimes D}\ar@{}[dr]|{\mathstrut\raisebox{2ex}{\textrm{ }\textrm{ }\ensuremath{(4)}\kern0.3em }} & D\otimes C_{2}\ar[d]^{f_{1}\otimes C_{2}}\\
C_{1}\otimes D\ar[r]^{C_{1}\otimes f_{2}}\ar[d]_{C_{1}\otimes f_{1}}\ar@{}[dr]|{\mathstrut\raisebox{2ex}{\textrm{ }\textrm{ }\ensuremath{(5)}\kern0.3em }} & C_{1}\otimes C_{2}\ar[d]^{C_{1}\otimes\tau_{C_{2}}}\\
C_{1}\otimes C_{1}\ar[r]_{C_{1}\otimes\tau_{C_{1}}} & C_{1}\otimes\top
}
\]
}where $(4)$ commutes for bifunctoriality. Diagram $(5)$ commutes
by taking a monoidal natural transformation $C_{1}\otimes\tau$ giving,
for $f_{1}$ and $f_{2}$ like above:{\small{}
\[
\xymatrix{C_{1}\otimes id(D)\ar[r]^{C_{1}\otimes\tau_{D}}\ar[d]_{C_{1}\otimes id(f_{1})} & C_{1}\otimes t(D)\ar[d]^{C_{1}\otimes t(f_{1})}\\
C_{1}\otimes id(C_{1})\ar[r]_{C_{1}\otimes\tau_{C_{1}}} & C_{1}\otimes t(C_{1})
}
\qquad\xymatrix{C_{1}\otimes id(D)\ar[r]^{C_{1}\otimes\tau_{D}}\ar[d]_{C_{1}\otimes id(f_{2})} & C_{1}\otimes t(D)\ar[d]^{C_{1}\otimes t(f_{2})}\\
C_{1}\otimes id(C_{2})\ar[r]_{C_{1}\otimes\tau_{C_{2}}} & C_{1}\otimes t(C_{2})
}
\]
}giving {\small{}
\[
\xymatrix{C_{1}\otimes D\ar[r]^{C_{1}\otimes\tau_{D}}\ar[d]_{C_{1}\otimes f_{1}} & C_{1}\otimes\top\ar[d]^{C_{1}\otimes\top}\\
C_{1}\otimes C_{1}\ar[r]_{C_{1}\otimes\tau_{C_{1}}} & C_{1}\otimes\top
}
\qquad\xymatrix{C_{1}\otimes D\ar[r]^{C_{1}\otimes\tau_{D}}\ar[d]_{C_{1}\otimes f_{2}} & C_{1}\otimes\top\ar[d]^{C_{1}\otimes\top}\\
C_{1}\otimes C_{2}\ar[r]_{C_{1}\otimes\tau_{C_{2}}} & C_{1}\otimes\top
}
\]
}both commuting for naturality. 

Then the $(1),(2),(3)$-diagram (together with its analogous for $C_{2}$)
is a cartesian product diagram where projections are $r\circ(C_{1}\otimes\tau_{C_{2}})$
and $l\circ(\tau_{C_{1}}\otimes C_{2})$ and the uniqueness is obvious
given $f_{1}$ and $f_{2}$.
\end{proof}

\section*{Appendix 4}

Let $\mathcal{SFR}^{n}$ be a subcategory of $\mathcal{FR}^{n}$ in
which we avoid any application of $T$ over the objects and morphisms
of $\mathcal{FR}^{n}$. We define some of the objects in $\mathcal{SFR}^{n}$
by means of $G$ as happens in the case of $\mathcal{FR}^{n}$ but
will make use of endofunctors $T$ only for the introduction of a
bounding condition in the recursion schemes used in $\mathcal{SFR}^{n}$. 

We introduce in the squares below the description of $\mathcal{SFR}^{n}$
in the form of a language for its objects and morphisms of $\mathcal{SFR}^{n}$.
The rules into the squares are subject to the following conventions: 
\begin{itemize}
\item we have omitted defining symmetric monoidal category rules (identity,
associativity as well as coherence diagrams)
\item $X,Y,Z$ and $W$ denote whatever object 
\item $f$ and $g$ denote whatever morphism
\item $a,l,\sigma$ denote the natural isomorphisms of the \emph{SM} structure
\item subindex $k$ will range between $0$ and $n-2$ into the squares
when no other indication is given.%
\footnote{Due to the syntactical behaviour of the definitions above, we should
mention that $d_{k}$ applied to $0_{k+1}$ gives $0_{k}$ and, analogously,
for every other natural number in a level $k+1$ it assigns the same
number in level $k$.%
}
\end{itemize}
{\scriptsize{}}%
\doublebox{\begin{minipage}[t]{1\columnwidth}%
\begin{enumerate}
\item \textbf{\footnotesize{}\uline{Objects}}{\footnotesize \par}

$ $
\begin{enumerate}
\item \textbf{\footnotesize{}\uline{$\mathbf{initial\; objects}$}}{\footnotesize \par}

\begin{center}
{\footnotesize{}$\infer[\mathbf{\top-object}]{\top}{\qquad}\qquad$$\infer[N_{k}\mathbf{-object}]{N_{k}}{\qquad}$}
\par\end{center}{\footnotesize \par}

\item \textbf{\footnotesize{}\uline{$\mathbf{generation\; of\; objects}$}}{\footnotesize \par}

\begin{center}
{\footnotesize{}$\infer[\mathbf{\otimes-object}]{X\otimes Y}{X\quad Y}$}
\par\end{center}{\footnotesize \par}

\end{enumerate}
\item \textbf{\footnotesize{}\uline{Arrows}}{\footnotesize \par}

$ $
\begin{enumerate}
\item \textbf{\footnotesize{}\uline{$\mathbf{initial\: arrows}$}}{\footnotesize \par}

{\footnotesize{}
\[
\infer[\mathbf{identity}]{id_{X}:X\rightarrow X}{X}\qquad\infer[\mathbf{zero}]{0_{k}:\top\rightarrow N_{k}}{}\qquad\infer[\mathbf{successor}]{s_{k}:N_{k}\rightarrow N_{k}}{}
\]
\[
\infer[\mathbf{eraser}]{\tau_{X}:X\rightarrow\top}{X}\qquad\infer[\mathbf{duplication}]{\delta_{X}:X\rightarrow X\otimes X}{X}
\]
}{\footnotesize \par}

{\footnotesize{}$ $
\[
\infer[\mathbf{drop}]{d_{k}:N_{k+1}\rightarrow N_{k}}{}
\]
}{\footnotesize \par}

\item \textbf{\footnotesize{}\uline{$\mathbf{generation\; of\; arrows\; from\; arrows}$}}{\footnotesize \par}

{\footnotesize{}
\[
\infer[\mathbf{composition}]{g\circ f:X\rightarrow Z}{f:X\rightarrow Y\quad g:Y\rightarrow Z}\qquad\infer[\mathbf{\otimes-arrow}]{g\otimes f:X\otimes Z\rightarrow Y\otimes W}{f:X\rightarrow Y\quad g:Z\rightarrow W}
\]
}{\footnotesize \par}

\item \textbf{\footnotesize{}\uline{$\mathbf{generation\; of\; arrows\; from\; objects\; and\; natural\: isomorphisms}$}}{\footnotesize \par}

{\footnotesize{}
\[
\infer[\mathbf{left}]{l:\top\otimes X\rightarrow X}{X}\quad\infer[\mathbf{symmetry}]{\sigma:X\otimes Y\rightarrow Y\otimes X}{X\quad Y}
\]
}{\footnotesize \par}

\end{enumerate}

{\footnotesize{}
\[
\infer[\mathbf{associativity}]{a:(X\otimes Y)\otimes Z\rightarrow X\otimes(Y\otimes Z)}{X\quad Y\quad Z}
\]
}{\footnotesize \par}

\item \textbf{\footnotesize{}\uline{Flat Recursion}}{\footnotesize \par}

\noindent {\footnotesize{}
\[
\infer[\mathbf{F}\mathbf{R}_{k}]{FR_{k}(g,h):N_{k}\otimes X\rightarrow Y}{g:X\rightarrow Y\quad h:N_{k}\otimes X\rightarrow Y}
\]
}{\footnotesize \par}

\noindent {\footnotesize{}$\textrm{where \ensuremath{X\textrm{ and }Y}are in the form \ensuremath{N_{k}^{\alpha}}}$}\end{enumerate}
\end{minipage}}{\scriptsize \par}

$ $

We now define some assignations which turn out to be, respectively,
\emph{SM} endofunctors and \emph{SM} natural transformations.
\begin{enumerate}
\item Let be {\footnotesize{}
\[
T_{k}X=\begin{cases}
\top & \textrm{if \ensuremath{X=N_{0}\textrm{ and }k=0}}\\
N_{k-1} & \textrm{if \ensuremath{X=N_{k}\textrm{ and }k\neq0}}\\
T_{k}Y\otimes T_{k}Z & \textrm{if }X=Y\otimes Z\\
X & \textrm{ otherwise }
\end{cases}\qquad G_{k}X=\begin{cases}
N_{k+1} & \textrm{if \ensuremath{X=N_{k}}}\\
G_{k}Y\otimes G_{k}Z & \textrm{if }X=Y\otimes Z\\
X & \textrm{ otherwise }
\end{cases}
\]
}and {\footnotesize{}
\[
T_{k}f=\begin{cases}
id_{\top} & \textrm{if \ensuremath{f=0_{0}\textrm{ or }s_{0}}}\\
0_{k-1} & \textrm{if \ensuremath{f=0_{k}\textrm{ and }k\neq0}}\\
s_{k-1} & \textrm{if \ensuremath{f=s_{k}}}\textrm{ and }k\neq0\\
\tau_{T_{k}X} & \textrm{if \ensuremath{f=\tau_{X}}}\\
id_{N_{k-1}} & \textrm{if \ensuremath{f=d_{k-1}}}\textrm{ and }k\neq0\\
\tau{}_{N_{1}} & \textrm{if \ensuremath{f=d_{0}}}\textrm{ and }k=0\\
\delta_{T_{k}X} & \textrm{if \ensuremath{f=\delta_{X}}}\\
d_{k-1}\circ d_{k} & \textrm{if \ensuremath{f=d_{k}}}\textrm{ and }k\neq0\\
T_{k}g\circ T_{k}h & \textrm{if }f=g\circ h\\
T_{k}g\otimes T_{k}h & \textrm{if }f=g\otimes h\\
f & \textrm{otherwise}
\end{cases}\qquad G_{k}f=\begin{cases}
0_{k+1} & \textrm{if \ensuremath{f=0_{k}}}\\
s_{k+1} & \textrm{if \ensuremath{f=s_{k}}}\\
\tau_{G_{k}X} & \textrm{if \ensuremath{f=\tau_{X}}}\\
\delta_{G_{k}X} & \textrm{if \ensuremath{f=\delta_{X}}}\\
id_{N_{k+1}} & \textrm{if \ensuremath{f=d_{k}}}\\
d_{k-1}\circ d_{k} & \textrm{if \ensuremath{f=d_{k-1}}}\\
G_{k}g\circ G_{k}h & \textrm{if }f=g\circ h\\
T_{k}g\otimes T_{k}h & \textrm{if }f=g\otimes h\\
f & \textrm{otherwise}
\end{cases}
\]
}for $f:X\rightarrow Y$ and for each $k=0,1,...,n-2$.
\item We denote by $\epsilon_{k}:G_{k}\Longrightarrow id$ and $\eta_{k}:id\Longrightarrow T_{k}$
some assignations%
\footnote{They are useful to get morphisms between objects in different levels.%
} {\footnotesize{}
\[
\epsilon_{k}X=\begin{cases}
\top & \textrm{if }X=\top\\
d_{k} & \textrm{if \ensuremath{X=N{}_{k}}}\\
\epsilon_{k}Y\otimes\epsilon_{k}Z & \textrm{if }X=Y\otimes Z\\
id_{X} & \textrm{otherwise}
\end{cases}\qquad\eta_{k}X=\begin{cases}
\top & \textrm{if }X=\top\\
d_{k-1} & \textrm{if \ensuremath{X=N{}_{k}}}\textrm{ and }k\neq0\\
\tau{}_{N_{0}} & \textrm{if \ensuremath{X=N_{0}}}\textrm{ and }k=0\\
\eta_{k}Y\otimes\eta_{k}Z & \textrm{if }X=Y\otimes Z\\
id_{X} & \textrm{otherwise}
\end{cases}
\]
}{\footnotesize \par}
\end{enumerate}
It easy to see that $T_{k}$ and $G_{k}$ are endofunctors and $\epsilon_{k}$
and $\eta_{k}$ are natural transformations in the free \emph{SM}
category defined by the rules above for each $k=0,1,...,n-2$. 

$ $

\noindent {\scriptsize{}}%
\doublebox{\begin{minipage}[t]{1\columnwidth}%
5.\textbf{\footnotesize{} }\textbf{\footnotesize{}\uline{Raising
arrows}}{\footnotesize \par}

\noindent {\footnotesize{}
\[
\infer[\mathbf{G_{k}-arrow}]{G_{k}f:G_{k}X\rightarrow G_{k}Y}{f:X\rightarrow Y}
\]
}{\footnotesize \par}

\noindent 6.\textbf{\footnotesize{} }\textbf{\footnotesize{}\uline{Safe
Recursion}}{\footnotesize{}
\[
\infer[\mathbf{SRR_{k}}]{SRR_{k}(g,h):N_{k+1}\otimes X\rightarrow Y}{g:X\rightarrow Y\quad h:Y\rightarrow Y}
\]
}{\footnotesize \par}

\noindent {\footnotesize{}$\quad\;\textrm{where }T_{k}...T_{0}Y\textrm{ is isomorphic to }\top$}{\footnotesize \par}

\noindent 7.\textbf{\footnotesize{} }\textbf{\footnotesize{}\uline{Safe
Dependent Recursion}}{\footnotesize \par}

\noindent {\footnotesize{}
\[
\infer[\mathbf{SDR_{k}}]{SDR_{k}(g,h):N_{k+1}\otimes X\rightarrow Y}{g:X\rightarrow Y\quad h:(N_{k+1}\otimes X)\otimes Y\longrightarrow Y}
\]
}{\footnotesize \par}

\noindent {\footnotesize{}$\quad\;\textrm{where }T_{k}...T_{0}Y\textrm{ is isomorphic to }\top$}%
\end{minipage}}
\end{document}